\documentclass[11pt]{article}
\usepackage{amssymb}
\usepackage{amsmath}
\usepackage{amsthm}
\usepackage{url}
\usepackage{pslatex}
\usepackage{subfigure}
\usepackage{times}
\usepackage[pdftex]{graphicx}
\usepackage{multirow}

\newtheorem{thm}{Theorem}
\newtheorem{lemma}[thm]{Lemma}
\newtheorem{prop}[thm]{Proposition}
\newtheorem{cor}[thm]{Corollary}

\theoremstyle{definition}

\renewcommand{\emptyset}{\varnothing}

\newcommand{\bb}{\mathbf{b}}
\newcommand{\uu}{\mathbf{u}}
\newcommand{\vv}{\mathbf{v}}

\newcommand{\RR}{\mathbb{R}}

\DeclareMathOperator{\rank}{rank}
\DeclareMathOperator{\detnum}{Det}
\DeclareMathOperator{\Aut}{Aut}
\DeclareMathOperator{\Sym}{Sym}
\DeclareMathOperator{\STS}{STS}

\begin{document}

\title{Resolving sets for Johnson and Kneser graphs\thanks{Partially supported by the ESF EUROCORES programme EuroGIGA  ComPoSe IP04  MICINN Project EUI-EURC-2011-4306 and projects MTM2008-05866-C03-01, JA-FQM164 and JA-FQM305.  R.~F.~Bailey acknowledges support from a PIMS Postdoctoral Fellowship and the University of Regina Research Trust Fund.  K.~Meagher acknowledges support from an NSERC Discovery Grant.}}

\author{Robert F.~Bailey\footnotemark[1] \and Jos\'e C\'aceres\footnotemark[2] \and Delia Garijo\footnotemark[3] \and Antonio Gonz\'alez\footnotemark[3] \and Alberto M\'arquez\footnotemark[3] \and Karen Meagher\footnotemark[4] \and Mar\'{\i}a Luz Puertas\footnotemark[2]}

\maketitle

\footnotetext[1]{Corresponding author.  Present address: Department of Mathematics, Ryerson University, 350 Victoria St., Toronto, Ontario M5B 2K3, Canada.  E-mail: \texttt{robert.bailey@ryerson.ca}}
\footnotetext[2]{Departamento de Estad\'{\i}stica y Matem\'atica Aplicada, Universidad de Almer\'{\i}a, Ctra.\ Sacramento s/n, Almer\'{\i}a 04120, Spain. E-mail:\texttt{\{jcaceres, mpuertas\}@ual.es}}
\footnotetext[3]{Departamento de Matem\'atica Aplicada I, Universidad de Sevilla, Av.\ Reina Mercedes s/n, Sevilla 41012, Spain. E-mail: \texttt{\{dgarijo, gonzalezh, almar\}@us.es}}
\footnotetext[4]{Department of Mathematics and Statistics, University of Regina, 3737 Wascana Parkway, Regina, Saskatchewan S4S 0A2, Canada.  E-mail: \texttt{karen.meagher@uregina.ca}}

\begin{abstract} 
A set of vertices $S$ in a graph $G$ is a {\em resolving set} for $G$ if, for any two vertices $u,v$, there exists $x\in S$ such that the distances $d(u,x) \neq d(v,x)$.  In this paper, we consider the Johnson graphs $J(n,k)$ and Kneser graphs $K(n,k)$, and obtain various constructions of resolving sets for these graphs.  As well as general constructions, we show that various interesting combinatorial objects can be used to obtain resolving sets in these graphs, including (for Johnson graphs) projective planes and symmetric designs, as well as (for Kneser graphs) partial geometries, Hadamard matrices, Steiner systems and toroidal grids.
\end{abstract}

\section{Introduction and preliminaries}
\label{sec:Introduction}
In this paper, we consider graphs $G=(V(G),E(G))$ that are finite, simple and connected. As usual, the distance between two vertices $u$ and $v$ is denoted by $d_G(u,v)$, or simply $d(u,v)$ if the graph $G$ is clear. A vertex $x\in V(G)$ is said to \emph{resolve} a pair $u, v\in V(G)$ if $d_G(u,x)\neq d_G(v,x)$. A set $S\subseteq V(G)$ is a \emph{resolving set} for $G$ if any pair of vertices of $G$ can be resolved by some vertex in $S$. If the set $S$ is as small as possible, then it is called a \emph{metric basis} and its cardinality $\beta(G)$ is the \emph{metric dimension} of the graph $G$.

Metric bases and resolving sets were first introduced to the graph theory literature in the 1970s by Slater~\cite{slater} and independently by Harary and Melter~\cite{hararymelter}. 
(However, the definition of a metric basis for an arbitrary metric space was known in the geometry literature at least 20 years earlier: see Blumenthal \cite[Definition 39.1]{Blumenthal}, for instance.)
In his seminal paper, Slater mentioned the following potential application: a moving point in a graph may be located by finding the distances from the point to a collection of sonar or LORAN stations which have been judiciously positioned in the graph.  

Subsequently, many other applications of resolving sets and metric dimension have appeared in the literature.  For example, the study of resolvability in hypercubes is closely related to a coin-weighing problem (see~\cite{sebotannier} for details); strategies for the {\em Mastermind} game use resolving sets in Hamming graphs~\cite{chvatal}; resolving sets in triangular, rectangular and hexagonal grids have been proposed to study digital images~\cite{digital}; a method based on resolving sets for differentiating substances with the same chemical formula is given in~\cite{chaerojohnoeller}.  Mathematical applications of closely-related parameters were given by Babai in the study of the graph isomorphism problem~\cite{Babai80} and in obtaining bounds on the possible orders of primitive permutation groups~\cite{Babai81} (see also~\cite{bsmd}).

Since the problem of computing the metric dimension of a graph is NP-complete (see~\cite{landmarks}), many efforts have been focused on finding either exact values or good bounds for the metric dimension of certain classes of graphs.  Examples include trees~\cite{hararymelter,slater}, wheels~\cite{wheels}, unicyclic graphs~\cite{poissonzhang}, Cayley digraphs~\cite{cayley} and cartesian products~\cite{doubleresolving}, among others. 

A lower bound on the metric dimension of a graph $G$ can be obtained by considering its automorphism group $\Aut(G)$.  A {\em base} for a group acting on a set is a collection of points, chosen so that the only group element fixing all of those points is the identity element; equivalently, every group element is uniquely specified by its action on those points.  (See~\cite{Cameron99} for more background on bases.)  Recently, in the case where the group is the automorphism group of a graph $G$, bases have been referred to as {\em determining sets} for $G$, and the least cardinality of a base for $\Aut(G)$ has become known as the {\em determining number} of $G$, denoted $\detnum(G)$ (see~\cite{Boutin06}).  It is straightforward to verify the following result (see, for instance, \mbox{\cite[Proposition 3.8]{bsmd}}).

\begin{prop} \label{prop:detmet}
For any finite, connected graph~$G$, we have $\detnum(G)\leq \beta(G)$.
\end{prop}

We refer the reader to~\cite{bsmd, determining} for further information on the relationship between the two parameters.

In this paper, we are interested in the metric dimension of Johnson and Kneser graphs, which we now introduce.

\subsection{Johnson and Kneser graphs}

The \emph{Kneser graph} $K(n,k)$ (where $n>k$) has the collection $\binom{[n]}{k}$ of all $k$-subsets of the $n$-set $[n]=\{1,\ldots, n\}$ as vertices, and edges connecting disjoint subsets.  As an example, the Petersen graph is the Kneser graph $K(5,2)$. Like Kneser graphs, the vertices of the \emph{Johnson graph} $J(n,k)$, with $n>k$, are the $k$-subsets of $[n]$, but two $k$-subsets are adjacent when their intersection has size $k-1$.

%

It is easy to see that the Kneser graph $K(n,k)$ is connected if and only if $n>2k$: if $n<2k$, there are no edges, while if $n=2k$, the Kneser graph is a perfect matching.  Also, it is not difficult to show that the Johnson graphs $J(n,k)$ and $J(n,n-k)$ are isomorphic.  Consequently, in the remainder of the paper, we shall only consider Kneser graphs with $n>2k$ and Johnson graphs with $n\geq 2k$.

A consequence of the definition is that in the Johnson graph $J(n,k)$ there is a one-to-one correspondence between intersection sizes and distances: specifically, the distance between two vertices $U$ and $W$ in $J(n,k)$ is given by
\begin{equation}
d(U,W)=|U\setminus W|=|W\setminus U|=k-|U\cap W|
\label{johnsondistance}.
\end{equation}
From this, it is clear that $J(n,k)$ has diameter~$k$.  Furthermore, one can show that the Johnson graph $J(n,k)$ is {\em distance-transitive}, i.e.\ for any vertices $U,W,X,Y$ with $d(U,W)=d(X,Y)$, there is an automorphism mapping $U$ to $X$ and $W$ to $Y$ (see~\cite{distance-regular} for more details).  In general, Kneser graphs do not have this property, as the correspondence between distances and intersection sizes does not arise.  However, there are two exceptional families, and both are ``extreme'' cases.  First, the Kneser graph $K(n,2)$ is the complement of the corresponding Johnson graph $J(n,2)$, and both graphs have diameter~2, so if $d_{K(n,2)}(U,W)=1$ then $d_{J(n,2)}(U,W)=2$, and vice-versa.  Secondly, there is the Kneser graph $K(2k+1,k)$ (known as the {\em Odd graph}: see \cite{Biggs79} for details).  The notation $O_{k+1}$ is often used to denote this graph, with the subscript $k+1$ being chosen as it is the valency of the graph; this family includes the Petersen graph as $O_3$.  The distance between two vertices in an Odd graph is determined exactly by the size of the intersection of the corresponding $k$-subsets, but by a different rule:
\begin{eqnarray*}
d(U,W) = 2r   & \iff & |U\cap W| = k-r; \\
d(U,W) = 2r+1 & \iff & |U\cap W| = r.
\end{eqnarray*}

In general, the distance between two vertices of a Kneser graph $K(n,k)$ is specified by the size of the intersection of the corresponding $k$-subsets (but not with a one-to-one correspondence).  If $n\geq 3k-1$, it is not difficult to see that two non-adjacent vertices of $K(n,k)$ share a common neighbour, and thus the distance between vertices $U$ and $W$ is either 1 or 2, depending on whether $U\cap W$ is empty or not.  More generally, if we write $n=2k+b$, it was shown in~\cite{valencia-pabon} that distances in $K(2k+b,k)$ are given by the following formula:
\begin{equation}
d(U,W)=\min\left\{2\left\lceil\frac{k-s}{b}\right\rceil, \, 2\left\lceil\frac{s}{b}\right\rceil+1\right\}\label{kneserdistance}
\end{equation}
for $U,W\in V(K(2k+b,k))$ and $s=|U\cap W|$.


In this paper, we are concerned with constructing resolving sets for Johnson and Kneser graphs.  To begin, we show that resolving sets for the two families of graphs are related in a straightforward way.

\begin{lemma}\label{imp}
Suppose $n>2k$.  Any resolving set $\mathcal{S}$ for the Kneser graph $K(n,k)$ is a resolving set for $J(n,k)$.  Thus $\beta(J(n,k))\leq\beta(K(n,k))$.
\end{lemma}

\begin{proof}
Suppose that the vertex $X$ in the Kneser graph $K(n,k)$ resolves the pair $U,W\in V(K(n,k))$. Clearly then $|X\cap U|\neq|X\cap W|$. By Equation~\ref{johnsondistance}, $U$ and $W$ are also resolved by $X$ in $J(n,k)$, and therefore the result follows.
\end{proof}

The converse of this lemma is not true in general, apart from the two exceptional families of Kneser graphs listed above, 
namely $K(n,2)$ and $K(2k+1,k)$.  In the first of those cases, any resolving set for $J(n,2)$ is also a resolving set for $K(n,2)$, and hence $\beta(J(n,2))=\beta(K(n,2))$; likewise, any resolving set for $J(2k+1,k)$ is also a resolving set for the Odd graph $O_{k+1}=K(2k+1,k)$, and hence $\beta(J(2k+1,k))=\beta(K(2k+1,k))$.

For $n>2k$, the Johnson graph $J(n,k)$ and Kneser graph $K(n,k)$ have the same automorphism group, namely the symmetric group~$\Sym(n)$ in its action on the $k$-subsets of $[n]$ (see~\cite[Sections 2.5 and 3.8]{bsmd}).  (If $n=2k$, then $\Aut(J(2k,k))\cong \Sym(2k)\times \mathbb{Z}_2$: the extra automorphisms arise from being able to interchange a $k$-subset with its complement.)  Thus, for $n>2k$, $\detnum(J(n,k))=\detnum(K(n,k))$.  
A summary of results about $\detnum(J(n,k))$ and $\detnum(K(n,k))$ can be found in \cite[Section 2.5]{bsmd};
in particular, in~\cite{kneserdet} 
the following result was obtained.

\begin{thm}[
{C\'aceres {\em et al.}\ \cite{kneserdet}}%
]
\label{thm:kneserdet}
Suppose $n>\binom{k+1}{2}$, and let $d$ be an integer such that $3\leq k+1 \leq d$.  Then whenever the inequality 
\[ \left\lfloor \frac{(d-1)(k+1)}{2} \right\rfloor < n-1 \leq \left\lfloor \frac{d(k+1)}{2} \right\rfloor \]
is satisfied, it follows that $\detnum(J(n,k))=\detnum(K(n,k))=d$.
\end{thm}

By Proposition~\ref{prop:detmet}, these provide a lower bound of approximately $2n/k$ on the metric dimension of these graphs.  We note that for fixed values of $k$, this lower bound is linear in $n$.

The metric dimension of the Johnson graph $J(n,2)$, and thus also the Kneser graph $K(n,2)$, were determined precisely in~\cite{bsmd}: the values depend on congruence classes modulo~$3$.
\begin{thm}[
{\cite[Corollary 3.33]{bsmd}}%
] 
\label{th:k=2}
Suppose $n\geq 6$.  Then for the metric dimension of the Johnson graph $J(n,2)$ and Kneser graph $K(n,2)$, where $n\equiv i \pmod 3$ (for $i=0,1,2$), we have $\beta(J(n,2))=\beta(K(n,2))=\frac{2}{3}(n-i)+i$.
\end{thm}
In fact, for $n\equiv 0 \pmod 3$, equality is achieved in Proposition~\ref{prop:detmet}, while in the other cases we have a difference of~$1$ between the determining number and metric dimension (see~\cite{bsmd} for details).

Our goal in this paper is not to obtain exact values for the metric dimension of Johnson and Kneser graphs, but rather to (i) give explicit constructions of resolving sets, and (ii) demonstrate how various interesting combinatorial and geometric structures may be used as resolving sets for these graphs.  In particular, some of our constructions provide good upper bounds on the metric dimension of $J(n,k)$ and/or $K(n,k)$.  

The remainder of the paper is organized into four sections: in Section~\ref{sec:general} we give some general constructions; Sections~\ref{sec:johnson} and~\ref{sec:kneser} are devoted to Johnson and Kneser graphs, respectively; Section~\ref{sec:final} has some concluding remarks.

\section{General constructions: partitioning the set $[n]$} \label{sec:general}

In this section, we give some constructions for resolving sets of Johnson and Kneser graphs, for arbitrary values of $n$ and $k$.  Each of these constructions involves specifying an appropriate partition of the set $[n]$, and taking subsets of the parts as the vertices of a resolving set.  We give two related but different constructions of resolving sets, considering Johnson and Kneser graphs separately; however, in the case $k=2$, the two constructions coincide, and each generalizes the construction in~\cite{bsmd} which yields Theorem~\ref{th:k=2}.  We then give an improved construction for Kneser graphs of diameter~$3$.

\subsection{A partitioning construction for Johnson graphs}
\label{sec:sets}

Recall from Equation~\ref{johnsondistance} that the distance between two vertices $U$ and $W$ in the Johnson graph $J(n,k)$ is given by
\begin{equation*}
d(U,W)=|U\setminus W|=|W\setminus U|=k-|U\cap W|.
\end{equation*}
Thus, a vertex $X\in V(J(n,k))$ resolves the pair $U,W\in V(J(n,k))$ if and only if $|X\cap U|\neq|X\cap W|$, which is equivalent to $|X\cap (U\setminus W)| \neq|X\cap (W\setminus U)|$.
A straightforward consequence is the following lemma.

\begin{lemma}\label{carjohnson}
A set of vertices $\mathcal{S}$ is a resolving set for $J(n,k)$ if and only if for any two disjoint non-empty sets $U,W\subset [n]$ such that $|U|=|W|\leq k$, there exists a vertex $X\in \mathcal{S}$ satisfying $|X\cap U|\neq |X\cap W|$.
\end{lemma}

Our construction of a resolving set yields the following result.

\begin{thm}
\label{th:[n]}
For the Johnson graph $J(n,k)$ with $n\geq 2k$, we have that 
\[ \beta(J(n,k))\leq \left\lfloor \frac{k}{k+1}(n+1)\right\rfloor. \]
\end{thm}

\begin{proof}
As we have already noted, the case $k=2$ was considered in~\cite{bsmd} (see Theorem~\ref{th:k=2} above), so we will suppose that $k>2$. 
We will divide our construction into two separate cases.  First, we will assume that $n=r(k+1)$ for some positive integer $r$, as our construction is more straightforward in that situation; later, we will suppose otherwise.

Consider the set $[n]=\{1,\ldots n\}$, and partition it into $r$ subsets $[n]=N_1\cup \cdots \cup N_r$ where $N_i=\{(i-1)(k+1)+1,\ldots ,i(k+1)\}$ and $1\leq i\leq r$. For each $i\in \{1,\ldots, r\}$, let $\mathcal{S}_i$ be the set of all $k$-subsets of $N_i$, but with one arbitrarily-chosen set removed. Note that any $X\in \mathcal{S}_i$ can be specified by the unique element of $N_i$ which is not in $X$.  Our claim is that $\mathcal{S}=\mathcal{S}_1\cup \cdots \cup \mathcal{S}_r$ is a resolving set for $J(n,k)$.  

Let $U$ and $W$ be two distinct vertices of $J(n,k)$, and consider how they intersect with the sets $N_1,\ldots,N_r$.  They can be partitioned into $U=U_1\cup \cdots \cup U_r$ and $W=W_1\cup \cdots \cup W_r$, where $U_i=U\cap N_i$ and $W_i=W\cap N_i$; note that some of these intersections may be empty. Our goal is to find a vertex $X\in \mathcal{S}$ which resolves $U$ and $W$, that is, $|X\cap U|\neq |X\cap W|$.  Note that if $X\in \mathcal{S}_i$, we have $X\cap U = X\cap U_i$, so it suffices to show that $|X\cap U_i| \neq |X\cap W_i|$.

Since $U\neq W$, there exists an index $i\in \{1,\ldots ,r\}$ such that $U_i\neq W_i$.  Then the following possibilities may occur.

\begin{description}
\item[Case 1.] Suppose first that $U_i=\emptyset$ and $W_i\neq \emptyset$. In this case, there exists some $X\in \mathcal{S}_i$ which resolves $U$ and $W$, since $X\cap U_i=\emptyset $ and we can choose an $X$ so that $X\cap W_i\neq \emptyset$.

\item[Case 2.] Now suppose that both $U_i$ and $W_i$ are non-empty and have different sizes; without loss of generality, we may assume that $0<|U_i|<|W_i|$.  We may also assume that $|W_i|<k$, as otherwise, there exists $j\neq i$ where $W_j=\emptyset$ and $U_j\neq \emptyset$, where we can apply Case 1.  Pick an element $a\notin W_i$ so that $X=N_i\setminus\{a\} \in \mathcal{S}_i$ (such an element exists, since $|W_i|<k$); note that this implies $W_i\subset X$.  Then we have $|X\cap U_i|\leq |U_i| < |W_i| = |X \cap W_i|$, and thus $X$ resolves $U$ and $W$.

\item[Case 3.] Finally, suppose that $|U_i|=|W_i|$ and both are non-empty.  Then there exist elements $a\in W_i\setminus U_i$ and $b \in U_i\setminus W_i$.  Now, $X=N_i\setminus \{a\}$ resolves $U$ and $W$, since $|X\cap U_i| = |U_i|$, but $|X\cap W_i|=|W_i|-1$; similarly, $X'=N_i\setminus \{b\}$ resolves $U$ and $W$, since $|X'\cap W_i| = |W_i|$, but $|X'\cap U_i|=|U_i|-1$.  At least one of $X,X'\in\mathcal{S}_i$.
\end{description}

Since $\mathcal{S}$ is a resolving set for $J(n,k)$ with $kr$ elements and $r=\frac{n}{k+1}$, the result follows.

Now we consider the case where $n$ is not divisible by $k+1$, i.e.\ where \mbox{$n=r(k+1)+j$} with \mbox{$1\leq j \leq k$.}  In this case, we partition the set $[n]$ as follows: let \mbox{$[n]=N_1\cup\cdots \cup N_r\cup N^*$,} where $N_1,\ldots,N_r$ are as before, and where \mbox{$N^*=\{n-j+1,\ldots ,n\}$.}  Then let $\mathcal{S}'=\mathcal{S}\cup \mathcal{S}^*$, where $\mathcal{S}$ is as defined above, and where the set $\mathcal{S}^*$ contains all sets of the form \mbox{$\{1,2,\ldots,k-1\}\cup\{x\}$,} for $x \in N^*$.  We claim that $\mathcal{S}'$ is a resolving set for $J(n,k)$.

Let $U$ and $W$ be two distinct vertices of $J(n,k)$.  Similar to the above, we partition $U$ into $U_1\cup \cdots \cup U_r \cup U^*$, where $U_i=U\cap N_i$ and $U^*=U\cap N^*$; likewise, we partition $W$ into $W=W_1\cup \cdots \cup W_r\cup W^*$.  If $U^*=W^*=\emptyset $, then $\mathcal{S}$ clearly resolves $U$ and $W$ by the arguments above; hence it suffices to consider the case in which either $U^*$ or $W^*$ are non-empty.

If $j=k$, it is possible that one of $U=N^*$ or $W=N^*$; without loss of generality assume that $U=N^*$, in which case any $X\in \mathcal{S}$ with $X\cap W\neq \emptyset$ resolves $U$ and $W$.  Otherwise, we must have that $U_i$ and $W_i$ are non-empty for some $i\in\{1,\ldots ,r\}$.

If $U_i\neq W_i$ for some index $i$, then the vertices can be resolved by some \mbox{$X\in \mathcal{S}$} as shown in Cases 1--3 above. Only when $U_i=W_i$ for all $i\in \{1,\ldots ,r\}$ is it necessary to choose a vertex from $\mathcal{S}^*$. However, since $U\neq W$, we have $U^*\neq W^*$ and both are non-empty.  Also, $|U^*|=|W^*|$, so there exists an element $x \in U^*\setminus W^*$.  Then $X=\{1,2,\ldots ,k-1\}\cup\{x\} \in \mathcal{S}^*$, with 
\begin{eqnarray*}
|X\cap U| & = & |\{1,\ldots,k-1\}\cap U|+1 \\
          & = & |\{1,\ldots,k-1\}\cap W|+1 \\
          & = & |X\cap W|+1.
\end{eqnarray*}
Hence $X$ resolves $U$ and $W$.

To conclude, $\mathcal{S}'=\mathcal{S}\cup\mathcal{S}^*$ is a resolving set for $J(n,k)$ of size 
\begin{eqnarray*}
rk+j & = & rk+\left\lfloor \frac{k}{k+1}(j+1)\right\rfloor \\
     & = & \left\lfloor \frac{k}{k+1}(n+1)\right\rfloor, 
\end{eqnarray*}
and the proof is complete.
\end{proof}

We remark that this construction has been adapted for the {\em Grassmann graphs} (see~\cite[Section 3]{grassmann}), the so-called ``$q$-analogue'' of the Johnson graphs, where the vertices are the $k$-dimensional subspaces of the vector space $\mathbb{F}_q^n$ and two vertices are adjacent if they intersect in a $(k-1)$-dimensional subspace.  
Subsequently, this construction was further adapted for various related classes of graphs, including the {\em bilinear forms graphs}, the {\em doubled Grassmann graphs} and {\em twisted Grassmann graphs}: see~\cite{fengwang,GuoWangLi}.  We also remark that Guo, Wang and Li have independently obtained the same bound as in Theorem~\ref{th:[n]} for the special case of $J(2k+1,k)$: see \cite[Theorem 2.2]{GuoWangLi}.

\subsection{A partitioning construction for Kneser graphs} 
\label{sec:kneserpart}

Inspired by the construction above which gives resolving sets for Johnson graphs, in this subsection we obtain a construction of resolving sets for Kneser graphs.  In a Kneser graph $K(n,k)$ with $n>2k$, we observe that for vertices $U,W$, if another vertex $X$ satisfies $X\cap U = \emptyset$ and $W\cap X\neq \emptyset$, then $X$ resolves $U$ and $W$ (since $X$ is adjacent to $U$ but not adjacent to $W$).  If $n\geq 3k-1$, this is the only way for a pair of vertices to be resolved (since $K(n,k)$ has diameter~$2$ in that case).  When $n<3k-1$, we give a variation on the construction below which gives an improved bound.

\begin{thm}\label{th:kneserparts}
For the Kneser graph $K(n,k)$ with $n > 2k$, we have that
\[
\beta(K(n,k)) \leq \left\lceil \frac{n}{2k-1} \right\rceil  \left( \binom{2k-1}{k}-1 \right).
\]
\end{thm}

\proof
Suppose $n=r(2k-1)+j$, where $0\leq j \leq 2k-2$.  Partition the set \mbox{$\{1,\ldots,n-j\}$} into parts $N_1,\ldots,N_r$, each of size $2k-1$.  For each $i \in \{1,\ldots,r\}$, let $\mathcal{S}_i$ be the collection of all $k$-subsets of $N_i$ but with one arbitrarily-chosen set removed; then let  $\mathcal{S}=\mathcal{S}_1\cup\cdots\cup \mathcal{S}_r$.  If $j\neq 0$, let $N_{r+1}=\{1,\ldots,2k-j-1\}\cup\{n-j+1,\ldots,n\}$, and let $\mathcal{T}$ denote the collection of all $k$-subsets from $N_{r+1}$ but with one arbitrarily-chosen set removed.  In this case, we let $\mathcal{S}=\mathcal{S}_1\cup\cdots\cup \mathcal{S}_r \cup \mathcal{T}$.

In either situation, it is clear that the size of $\mathcal{S}$ is
\[
\left\lceil \frac{n}{2k-1} \right\rceil  \left( \binom{2k-1}{k}-1 \right).
\]
We claim that $\mathcal{S}$ is a resolving set for $K(n,k)$. To prove this claim, we need to show for any distinct pair of $k$-subsets $U,W
\in V(K(n,k))$ that either one of $U$ or $W$ is in $\mathcal{S}$, or there is a set in $\mathcal{S}$ that intersects exactly one of $U$ and $W$.  

Now, if there exists an $i$ such that $U \cap N_i \neq \emptyset$ and $W \cap N_i = \emptyset$ (or conversely $U \cap N_i = \emptyset$ and $W \cap N_i \neq \emptyset$), then any $k$-subset of $N_i$ that intersects with $U$ will not intersect with $W$. The set $\mathcal{S}$ will certainly contain many such subsets of $N_i$.  

If the above does not hold, then for any $i$, if $U \cap N_i \neq \emptyset$ then $W \cap N_i \neq \emptyset$. Since $U \neq W$, there is an $i$ such that $U \cap N_i$ and $W \cap N_i$ are distinct, and both are non-empty. Now we consider three cases:

\begin{description}
\item[Case 1.] 
If $U \cap N_i =U$ and $W \cap N_i =W$ then at least one of $U$ and $W$ will be in $\mathcal{S}$.

\item[Case 2.] If $U \cap N_i =U$ and $|W \cap N_i| \leq k-1$ then there is another part $N_{i'}$ that intersects with $W$ but not $U$, and we are done.

\item[Case 3.] Assume $|U \cap N_i| \leq k-1$. Then $|W \cap N_i| \leq k-1$, since otherwise there would exist an $i'$ such that $U \cap N_{i'} \neq \emptyset$ and $W \cap N_{i'} = \emptyset$, and again we are done.

Since $U \cap N_i \neq W \cap N_i$, at least one element from the complement $\overline{U \cap N_i}$ is in $W$, and since $|\overline{U \cap N_i}| \geq k$, there is a $k$-subset of $\overline{U \cap N_i}$ that intersects $W$ but not $U$. Similarly, there is a $k$-subset of $\overline{W \cap N_i}$ that intersects $U$ but not $W$. At least one of these $k$-subsets is in $\mathcal{S}$.
\endproof
\end{description}

\subsection{An improved construction for Kneser graphs of diameter 3}
\label{sec:kneserpart2}

For $n<3k-1$, the diameter of the Kneser graph $K(n,k)$ is greater than~$2$, and consequently it should be possible to refine our construction from the previous subsection, in order to obtain smaller resolving sets when the diameter is larger.  When $\lfloor 5k/2 \rfloor \leq n \leq 3k-2$, it follows from~\cite{valencia-pabon} that $K(n,k)$ has diameter~$3$, and that for two vertices $U,W$ the distance between them in $K(n,k)$ is as follows:
\begin{eqnarray*}
d(U,W) = 0 & \iff & |U\cap W| = k; \\
d(U,W) = 1 & \iff & |U\cap W| = 0; \\
d(U,W) = 2 & \iff & 3k-n \leq |U\cap W| \leq k-1; \\
d(U,W) = 3 & \iff & 1 \leq |U\cap W| \leq 3k-n-1.
\end{eqnarray*}

\begin{thm}\label{kneserdiam3}
For the Kneser graph $K(n,k)$ where $n$ and $k$ are integers such that $\lfloor 5k/2 \rfloor \leq n \leq 3k-2$, we have
\[
\beta(K(n,k)) \leq 2 \binom{n-k}{k}.
\]
\end{thm}

\begin{proof}
Where $[n]=\{1,\ldots,n\}$, we define the overlapping subsets
\[
N_1 = \{1,2,\ldots,n-k\} , \quad N_2=\{k+1, k+2,\ldots,n\}.
\]
Then let $\mathcal{S}_i$ be the collection of all $k$-subsets of $N_i$, and set $\mathcal{S} = \mathcal{S}_1 \cup \mathcal{S}_2$. Clearly the size of $S$ is $2 \binom{n-k}{k}$, and we claim that $S$ is a resolving set for $K(n,k)$.  To do so, we must show that for any two distinct vertices $U,W$ of $K(n,k)$, there is a vertex $X\in\mathcal{S}$ satisfying $d(U,X)\neq d(W,X)$.

We remark that, since $n \in \{ \lfloor 5k/2 \rfloor ,\ldots,3k-2\}$, we have that \mbox{$n-k \in \{ \lfloor 3k/2 \rfloor ,\ldots, 2k-2\}$} and that $3k-n-1 \in
\{1,\ldots, \lceil k/2 \rceil -1 \}$.  Also, we observe that if either $U$ or $W$ is properly contained in either $N_1$ or $N_2$, then one of $U$ and $W$ belongs to $\mathcal{S}$, so we may assume otherwise.  For $i = 1,2$, we define $U_i = U \cap N_i$ and $W_i = W \cap N_i$; by our assumption, we have $|U_i| \leq k-1$ and $|W_i| \leq k-1$. Since $U$ and $W$ are distinct, $U_i \neq W_i$ for some $i$, so without loss of generality we will assume that $U_1 \neq W_1$. Once again, there are several cases to consider.

\begin{description}
\item[Case 1(a).] If $|U_1| \leq k/2$ and $W_1 \subset U_1$, then choose $X$ to be
   any $k$-subset of $N_1$ that contains one element from $U_1 \setminus W_1$ and 
   all other elements from $N_1 \setminus (U_1 \cup W_1)$ (this is possible since 
   $|N_1 \setminus(U_1 \cup W_1)| \geq k$). Then $d(U,X) = 3$ and $d(W,X) = 1$.
\item[Case 1(b).] If $|U_1| \leq k/2$ and $W_1 \not\subseteq U_1$, let $X$ be a 
   $k$-subset of $N_1 \setminus U_1$ that contains at least one element from 
   $W_1$. In this case, $d(U,X) = 1$ and $d(W,X) = 2$ or $3$.
\end{description}
Clearly, if $|W_1| \leq k/2$ then this case also holds.

\begin{description}
\item[Case 2.] Now we must suppose that $k/2< |U_1| \leq k-1$ and
  $k/2<|W_1| \leq k-1$; note that the lower bound also implies that
  $|U_1| > 3k-n-1$ and $|W_1| > 3k-n-1$.  Without loss of generality, we may
  assume that $|U_1| \geq |W_1|$.  There are three subcases to consider, 
  and in each of these we construct a vertex $X$ with $d(U,X) = 2$ and $d(W,X) =3$.
\begin{itemize}
  \item[(a)] If $|U_1 \cap W_1| = 3k-n-1$, define $X$ to be a $k$-subset containing 
  all of $U_1$ and $k - |U_1|$ elements from $N_1 \setminus (U_1 \cup W_1)$. For 
  such a set $X$ to exist, we need to show that $|N_1 \setminus (U_1 \cup W_1)|$ 
  is sufficiently large. This is straightforward since 
\begin{align*}
|N_1 \setminus (U_1 \cup W_1)| & = (n-k) - (|U_1| + |W_1| - |U_1 \cap W_1|)\\
 & \geq n-k - |U_1| - (k-1) + (3k-n-1)\\
 & = k - |U_1|.
\end{align*}

  \item[(b)] If $|U_1 \cap W_1| < 3k-n-1$, then we can choose $X$ to be a $k$-subset 
  that contains all of $U_1$, $(3k-n-1) - |U_1 \cap W_1|$ elements from $W_1 \setminus U_1$ 
  (this is possible since $|W_1| > 3k-n-1$) and $k - |U_1| -(3k-n-1 -|U_1 \cap W_1|)$
  elements from \mbox{$N_1 \setminus (U_1 \cup W_1)$.}  Again, for such a set $X$ to exist we need 
  to show that $|N_1 \setminus (U_1 \cup W_1)|$ is sufficiently large; this follows because
  \begin{align*}
|N_1 \backslash (U_1 \cup W_1)| 
& =  (n-k) - |U_1| - |W_1| + |U_1 \cap W_1|\\
& \geq  n-k - |U_1| - (k-1) + |U_1 \cap W_1|\\
& =  k - |U_1| - (3k-n-1) + |U_1 \cap W_1|.
\end{align*}

  \item[(c)] If $|U_1 \cap W_1| > 3k-n-1$ then we can set $X$ to be a $k$-subset with 
  $3k-n-1$ elements from $U_1 \cap W_1$, all of $U_1 \setminus W_1$ (this is not empty 
  since $|U_1| \geq |W_1|$) and $k - (3k-n-1) - |U_i \setminus W_i|$ elements from
  $N_1 \setminus (U_1 \cup W_1)$. To show that this last requirement can be met, consider
\begin{align*}
|N_1 \setminus (U_1 \cup W_1)| 
& =  (n-k) - |U_1\setminus W_1| - |W_1| \\
& \geq  n-k - |U_1\setminus W_1| - (k-1) \\
&= k - (3k-n-1) - |U_i \setminus W_i|.
\end{align*}
\end{itemize}
In all cases, we find that $d(U,X) = 2$ and $d(W,X) =3$, and thus $X$ resolves $U$ and $W$.
\end{description}
This completes the proof.
\end{proof}

\section{Resolving sets for Johnson graphs: an algebraic approach} \label{sec:johnson}

\subsection{A matrix method}
\label{sec:incidence matrices}

In this subsection, we introduce a useful technique based on incidence matrices that can be used to show that certain families of $k$-subsets of $[n]$ are resolving sets for the Johnson graph $J(n,k)$.
%

Let $S$ be a subset of $[n]$.  The \emph{incidence vector} of $S$ is the vector $(v_1,\ldots,v_n) \in\RR^n$ whose entries are
\[ v_i = \left\{ \begin{array}{ccl} 1 & \,\, & \textnormal{if $i\in S$,} \\ 0 && \textnormal{otherwise.} \end{array} \right. \]
Now suppose we have a family of subsets (or a \emph{set system}) $\mathcal{S}=\{S_1,\ldots,S_t\}$, where each $S_i$ is a subset of $[n]$ with a fixed cardinality.  Then the \emph{incidence matrix} of $\mathcal{S}$ is the $t\times n$ matrix whose rows are the incidence vectors of $S_1,\ldots,S_t$.

So given any subset of the vertex set of $J(n,k)$, we can write down an incidence matrix for it.  This approach gives a straightforward method of verifying that a given set system is a resolving set for $J(n,k)$, with the following lemma being a straightforward, yet crucial, observation.

\begin{lemma} \label{lemma:matrix}
Let $A$ be the incidence matrix of a set system $S_1,\ldots,S_t$ formed of subsets of $[n]$, and let $\vv=(v_1,\ldots,v_n)$ be the incidence vector of an arbitrary subset $U\subseteq [n]$.  Suppose $\bb=(b_1,\ldots,b_t)$ is the vector obtained as $A\vv=\bb$.  Then, for all $i$, we have
\[ b_i = |S_i \cap U|. \]
\end{lemma}


Lemma~\ref{lemma:matrix} gives us an algebraic definition of resolving sets for $J(n,k)$.  Let $\mathcal{S}=\{S_1,\ldots,S_t\}$ be a resolving set for $J(n,k)$.  Since $\mathcal{S}$ is a resolving set for $J(n,k)$, for any two $k$-subsets $U,W$ of $[n]$, there exists some $S_i\in\mathcal{S}$ with $|S_i\cap U| \neq |S_i\cap W|$.  Consequently, we have
\[ (|S_1\cap U|,|S_2\cap U|,\ldots,|S_t\cap U|) = (|S_1\cap W|,|S_2\cap W|,\ldots,|S_t\cap W|) \]
if and only if $U=W$.  Now let $M$ denote the set of incidence vectors of $k$-subsets of $[n]$, and suppose $A$ is the incidence matrix of $\mathcal{S}$.  From Lemma~\ref{lemma:matrix} it follows that for all $\uu,\vv\in M$, we have $A\uu=A\vv$ if and only if $\uu=\vv$.

Now, if the matrix~$A$ represents a linear transformation which is one-to-one, then we are guaranteed that $A\uu=A\vv$ if and only if $\uu=\vv$ for all $\uu,\vv\in\RR^n$, not just incidence vectors.  This leads us to the following result.

\begin{thm} \label{thm:matrix}
Suppose $\mathcal{S}$ is a family of $k$-subsets of $[n]$ whose incidence matrix has rank $n$.  Then $\mathcal{S}$ is a resolving set for the Johnson graph $J(n,k)$.
\end{thm}

\begin{proof}
Suppose that $|\mathcal{S}|=t$, and that $A$ is the incidence matrix of $\mathcal{S}$.  Since $\rank(A)=n$, it follows that $t\geq n$.  Let $\tau \, : \, \RR^n \to \RR^t$ be the linear transformation represented by the matrix $A$.  By the rank-nullity theorem, $\tau$ is one-to-one.  Thus for all vectors $\uu,\vv\in\RR^n$, we have $A\uu=A\vv$ if and only if $\uu=\vv$.  In particular, this holds for incidence vectors of $k$-subsets, so by the above argument, $\mathcal{S}$ is a resolving set for $J(n,k)$.
\end{proof}

If we happen to have a $n\times n$ incidence matrix with rank $n$, the matrix would have to be invertible.  As a corollary to the above, we show that such an invertible matrix always exists.

\begin{cor} \label{cor:matrix}
For any values of $n$ and $k$, the metric dimension of the Johnson graph $J(n,k)$ is at most $n$.
\end{cor}

\begin{proof}
To show this, we just need to exhibit a set system of size $n$ with an invertible incidence matrix, which we shall construct.  
As is usual, $I_n$ denotes the $n\times n$ identity matrix, and $J_n$ denotes the $n\times n$ matrix with all entries equal to 1.  Then let $A$ be the following $n\times n$ matrix:
\[ A = \left[ \begin{array}{c|c}
                &   \\
J_{k+1}-I_{k+1} & 0 \\
                &   \\ \hline
                &   \\
              B & I_{n-k-1} \\
                &
\end{array} \right], \,\,\, \textrm{ where } \,
 B = \begin{bmatrix} 1 & 1 & \cdots & 1 & 0 & 0 \\
                       1 & 1 & \cdots & 1 & 0 & 0 \\
                       \vdots & \vdots & \vdots & \vdots & \vdots & \vdots  \\
                       1 & 1 & \cdots & 1 & 0 & 0 \end{bmatrix}. \]
Clearly the rows of $A$ are 0-1 vectors of weight $k$.  Also, this matrix is clearly invertible, as 
its determinant is
\[ \det(A) = \det(J_{k-1}-I_{k-1}) \det(I_{n-k-1}) = (-1)^{k}k, \]
which is obviously not zero.

The set system which corresponds to this matrix is then
\[  { \{1,\ldots,k+1\} \choose k } \cup \left\{ \{1,\ldots,k-1,x\} \, : \, x\in \{k+2,\ldots,n\} \right\}. \]
\end{proof}

As an example, the following is a resolving set for $J(9,3)$ of size~$9$:
\[ \{1,2,3\}, \{1,2,4\}, \{1,2,5\}, \{1,2,6\}, \{1,2,7\}, \{1,2,8\}, \{1,2,9\}, \{1,3,4\}, \{2,3,4\}. \]

We remark that this approach has also been adapted for the Grassmann graphs: see~\cite[Theorem 5]{grassmann} for details.

\subsection{Symmetric designs}
\label{sec:symmetric designs}

We can use the approach developed in the previous subsection to demonstrate that a particularly interesting class of set systems provides resolving sets for $J(n,k)$ of size~$n$.

A \emph{$2$-design} with parameters $(n,k,\lambda)$ is a pair $(X,\mathcal{B})$, where $X$ is a set of $n$ \emph{points}, and $\mathcal{B}$ is a family of $k$-subsets of $X$, called \emph{blocks}, such that any pair of distinct points are contained in exactly $\lambda$ blocks. The \emph{incidence matrix} of a 2-design is the 0-1 matrix with rows indexed by the points and columns indexed by the blocks of the 2-design, where the $(p,B)$ entry is 1 if the point $p$ is in the block $B$ and $0$ otherwise.

A well-known result is Fisher's inequality (see~\cite[Theorem 1.9]{Lander83}), which asserts that the number of blocks is at least the number of points~$n$.  If the number of blocks is in fact equal to $n$, we have a \emph{symmetric design}.  If we have a symmetric design with parameters $(n,k,\lambda)$ and incidence matrix $A$, then $A^T$ must also be the incidence matrix of a symmetric design with those parameters.  Consequently, in a symmetric design, any pair of distinct blocks must intersect in exactly $\lambda$ points.  A table listing families of symmetric designs can be found in~\cite[Section II.6.9]{handbook}.

Incidence matrices are a powerful tool in the study of symmetric designs (see \cite{Lander83}, for instance).  The most well-known existence result for symmetric designs, the Bruck--Ryser--Chowla theorem (which gives strong necessary conditions for their existence: see~\cite[Theorem 2.1]{Lander83}), is obtained using them.  For our purpose, we can use incidence matrices to show the following.

\begin{thm} \label{thm:symmetricdesign}
The blocks of a symmetric design $\mathcal{D}$ with parameters $(n,k,\lambda)$ form a resolving set for $J(n,k)$.
\end{thm}

\begin{proof}
Suppose $A$ is the incidence matrix of $\mathcal{D}$.  By \cite[Proposition 1.2]{Lander83}, we have $|\det(A)| = k\sqrt{(k-\lambda)^{n-1}}$,
and this equals 0 if and only if $\lambda=k$.  However, in a symmetric design this can only happen if $n=k$ (see~\cite[Proposition~1.1]{Lander83}), which is trivial.

Hence $\mathcal{D}$ has an invertible incidence matrix, so by Theorem~\ref{thm:matrix} is a resolving set for $J(n,k)$.
\end{proof}

Three particular classes of symmetric designs are worth mentioning here.  First, symmetric designs with $\lambda=1$ are precisely the \emph{finite projective planes} \cite{HughesPiper73}.  For these to exist, we must have $n=q^2+q+1$ and $k=q+1$ for some positive integer $q$, which is called the \emph{order} of the projective plane.  Projective planes are known to exist for any prime-power order, and it is conjectured that these are the only orders possible.  The most famous example is the \emph{Fano plane} 
which is a symmetric design with parameters $(7,3,1)$, and thus can be used as a resolving set for the Johnson graph $J(7,3)$.  In Subsection~\ref{sec:partial geometries}, we shall see that (with the exception of the Fano plane) projective planes do not give resolving sets for Kneser graphs.

%
Symmetric designs with $\lambda=2$ are known as \emph{biplanes} \cite{Cameron73}.  For a biplane to exist, we must have $n={k \choose 2}+1$.  Unlike the case of projective planes, there are no known infinite families of biplanes.  In fact only 16 examples are known (see \cite{Kaski08}), the largest having $n=79$ points and blocks of size $k=13$.

Another important class of symmetric designs are those arising from Hadamard matrices, which will be discussed in subsection~\ref{sec:odd graphs} below.

\section{Resolving sets for Kneser graphs: combinatorial and \\geometric approaches} \label{sec:kneser}

In this section we discuss a number approaches to the construction of resolving sets for various classes of Kneser graphs.  These constructions, which may appear on the surface to be something of a ``mixed bag'', demonstrate the variety of techniques which may be employed.  Our constructions are inspired by finite and discrete geometry, as well as combinatorial design theory. We also discuss the implications of the algebraic techniques from the previous section for Kneser graphs.

\subsection{Partial geometries}
\label{sec:partial geometries}

A \emph{partial geometry} with parameters $(s,t,\alpha)$, or $pg(s,t,\alpha)$, is a pair $(\mathcal{P},\mathcal{L})$, consisting of a set of \emph{points} $\mathcal{P}$ and a set of \emph{lines} $\mathcal{L}$, satisfying the following conditions:
\begin{itemize}
\item[(i)] any line is incident with $s+1$ points, and the intersection of any two lines is at most a single point;
\item[(ii)] any point is incident with $t+1$ lines, and any two points with at most one line;
\item[(iii)] if the point $p$ and the line $L$ are not incident, then exactly $\alpha$ points of $L$ are collinear with $p$ (and so also exactly $\alpha$ lines incident with $p$ are concurrent with $L$).
\end{itemize}

This is a very general geometric structure, with many well-known objects (including projective and affine planes, generalized quadrangles, etc.) occurring as special cases.  (For additional background material about partial geometries, see~\cite{batten}).  
We remark that given any partial geometry $pg(s,t,\alpha)$, its \emph{dual} is a partial geometry $pg(t,s,\alpha)$, obtained by interchanging the roles of points and lines. Also, in a partial geometry $pg(s,t,\alpha)$, the number of points $v$ and the number of lines $b$ are given by 
\[ v= \frac{(s + 1)(st + \alpha)}{\alpha} \quad \textnormal{and} \quad b=\frac{(t + 1) (st + \alpha)}{\alpha}. \]
Our main result in this subsection, where we use partial geometries to obtain resolving sets for Kneser graphs, is as follows.

\begin{thm}\label{t>s}
Let $\Gamma$ be a partial geometry $pg(s,t,\alpha)$ with point set $\mathcal{P}$ and line set $\mathcal{L}$, and where $t>s$ . Then $\mathcal{L}$ is a resolving set for the Kneser graph $K(v,s+1)$.
\end{thm}

\begin{proof}
Let $\Gamma$ be the partial geometry $pg(s,t,\alpha)$ given by the set of lines $\mathcal{L}=\{L_1,\ldots,L_{b}\}$ over the set of
points $\mathcal{P}=\{1,\ldots,v\}$. Note that the lines can be viewed as vertices of the Kneser graph $K(v,s+1)$. Consider two distinct vertices
$U,W\in V(K(v,s+1))$, a point $p\in \mathcal{P}$ such that $p\in U\setminus W$ and the $t+1$ lines incident with $p$. Since any two of these lines only intersect in $p$  and $t>s$, there exists a line $L_i\in\mathcal{L}$ containing $p$ and not intersecting $W$.  Thus $d(L_i,U)>1$ and $d(L_i,W)=1$, and so $L_i$ resolves $U$ and $W$. Hence $\mathcal{L}$ is a resolving set for $K(v,s+1)$.
\end{proof}

Note that, by Lemma~\ref{imp}, the partial geometries of Theorem~\ref{t>s} are also resolving sets for the Johnson graphs $J(v,s+1)$.

Partial geometries where $\alpha=s+1=t$ are affine planes of order $s+1$. Since affine planes of order $q$ are known to exist whenever $q$ is a prime power (see~\cite{batten}), we have the following upper bound for the metric dimension of $K(q^2,q)$ for prime powers~$q$.

\begin{cor}
\label{affine}
If $q\geq 3$ is a prime power, then $\beta (K(q^2,q))\leq q^2+q$.
\end{cor}

\begin{proof}
Apply Theorem~\ref{t>s} for values $s=q-1$, $t=q$ and $\alpha=q$, noting that $v=q^2$.
\end{proof}

When $\alpha=1$, we obtain the so-called \emph{generalized quadrangles} (see for instance \cite{paynethas}). Their existence is known for many values of $(s,t)$, including the classical ones: $(q-1,q+1)$, $(q,q^2)$ and $(q^2,q^3)$ for a prime power $q$.  Thus, Theorem~\ref{t>s} gives upper bounds on the metric dimension of some further families of Kneser graphs, such as the following ones.
\begin{cor}
\label{gq}
If $q$ is a prime power, then:
\begin{itemize}
\item[{\rm (i)}] $\beta (K(q^3,q))\leq q^2(q+2)$;
\item[{\rm (ii)}] $\beta (K((q+1)(q^3+1),q+1))\leq (q^2+1)(q^3+1)$;
\item[{\rm (iii)}] $\beta (K((q^2+1)(q^5+1),q^2+1))\leq (q^3+1)(q^5+1)$.
\end{itemize}
\end{cor}

A partial geometry $pg(q,q,q+1)$ with $q\geq 2$ is a projective plane of order $q$. Since affine planes are resolving sets for an infinite family of Kneser graphs, it prompts the question of whether projective planes are also.  However, the next result shows that the answer to this question is negative.

\begin{prop} \label{th:notproj}
Given a projective plane of order $q>2$, the set $\mathcal{L}$ of lines does not resolve the Kneser graph $K(q^2+q+1,q+1)$.
\end{prop}

\begin{proof}
A projective plane of order $q>2$ has $v=q^2+q+1$ points and every line contains exactly $q+1$ points, so we are dealing with Kneser graphs $K(q^2+q+1,q+1)$. Clearly, the diameter of $K(q^2+q+1,q+1)$ is two: since $q>2$, we have $q^2+q+1 \geq 3q+2$.  Consider a line $L\in \mathcal{L}$ and two points $p,p'\in L$.  In a projective plane, there exist exactly $q+1$ distinct lines incident with $p$, say $\{L,L_1,...,L_q\}$, and exactly $q+1$ distinct lines incident with $p'$, say $\{L,L'_1,...,L'_q\}$.  Also, any two lines intersect, and so we can consider the set of points
$\{p_i=L_i\cap L'_i \, : \, i=1,\ldots,q\}$.  Note that these must all be distinct.

We will show that the vertices $U=\{p,p_1,...,p_q\}$ and $W=\{p',p_1,...,p_q\}$ are not resolved by any line of $\mathcal{L}$. Indeed, since the diameter of $K(q^2+q+1,q+1)$ is two, any line $X$ resolving $U$ and $W$ should intersect only one of these two vertices.  
Thus, $X$ must be disjoint from $\{p_1,...,p_q\}$, and contains only one of $p$ and $p'$.  So $L$ cannot resolve $U$ and $W$, and every line other than $L$ incident with either $p$ or $p'$ also intersects $\{p_1,...,p_q\}$. This proves that $\mathcal{L}$ is not a resolving set for $K(q^2+q+1,q+1)$.
\end{proof}

We remark that the above proof excludes the case of $q=2$, where the unique projective plane is the Fano plane.  It transpires that the Fano plane actually does give a resolving set for the Odd graph $K(7,3)$: this is discussed in the following subsection.

 \subsection{Odd graphs and Hadamard matrices}
\label{sec:odd graphs}

In Section~\ref{sec:Introduction}, we saw that for the Odd graph~$O_{k+1}$ (i.e.\ the Kneser graph $K(2k+1,k)$), any resolving set for the corresponding Johnson graph $J(2k+1,k)$ will also resolve~$O_{k+1}$.  Consequently, the results we obtained in the previous section can be applied directly to Odd graphs.  In particular, Corollary~\ref{cor:matrix} (using incidence matrices) implies that \mbox{$\beta(J(2k+1,k))=$} $\beta(O_{k+1})\leq 2k+1$, while Theorem~\ref{th:[n]} (using our ``partitioning'' construction) yields \mbox{$\beta(J(2k+1,k))=$} $\beta(O_{k+1})\leq 2k$. 

While incidence matrices give a (slightly) weaker bound here, there is however an interesting class of symmetric designs which can be used here.  A \emph{Hadamard matrix} is an $n\times n$ square matrix $H$ with entries $\pm 1$ and the property that $HH^T = nI_n$.  For such a matrix to exist, we must have $n=1$, $n=2$ or $n$ being a multiple of 4; it is conjectured that they exist for all such values, with the smallest size for which existence is unknown being $n=668$ (see~\cite{Hadamard}).  Any Hadamard matrix may be \emph{normalized} so that the first row and column have all entries~$+1$.  Given a normalized $4m\times 4m$ Hadamard matrix, by deleting the first row and column and replacing the entries $-1$ with $0$, one obtains the incidence matrix of a symmetric design with parameters $(4m-1,2m-1,m-1)$ (see~\cite[Section 1.2]{Lander83}), called a \emph{Hadamard design}.  Note that for $m=2$, the unique Hadamard design is the Fano plane, while for $m=3$, we obtain the unique biplane on 11 points.

In particular, where $k=2m-1$ and there exists a $4m\times 4m$ Hadamard matrix, Theorem~\ref{thm:symmetricdesign} shows that we can use a Hadamard design as a resolving set for \mbox{$J(2k+1,k)$.}  By the observation above, such a design may also be used as a resolving set for the Odd graph $O_{k+1}$.  As an example, the Fano plane is a resolving set for $K(7,3)=O_4$.

\subsection{Steiner systems}
\label{sec:steiner systems}

The Fano plane, which as we have seen is a resolving set for $K(7,3)$, is an example of an important class of combinatorial objects known as Steiner systems.  In this subsection, we show that these objects may be used as resolving sets for Kneser graphs more widely.

Let $n,k,t,\lambda$ be integers with $n> k > t >1$.  A {\em $t$-$(n,k,\lambda)$ design} (or a {\em $t$-design}) is a pair $(X,\mathcal{B})$, where $X$ is a set of $n$ \emph{points}, and $\mathcal{B}$ is a family of $k$-subsets of $X$, called \emph{blocks}, such that any $t$-subset of distinct points are contained in exactly $\lambda$ blocks.  From the definition, it follows that the number of blocks in a $t$-design is necessarily
\begin{equation} \label{blocksnumber}
\lambda \binom{n}{t} / \binom{k}{t}.
\end{equation} 
Usually we take $X=[n]$.  For $t=2$, we recover the definition of 2-designs from Section~\ref{sec:symmetric designs}.  A {\em Steiner system}~$S(t,k,n)$ is a $t$-design with $\lambda=1$, i.e.~any $t$-subset of $[n]$ is contained in exactly one block.  (See~\cite[Section II.5]{handbook} for more background on Steiner systems.)

Some important subclasses of Steiner systems are as follows: projective planes of order $q$, which are Steiner systems $S(2,q+1,q^2+q+1)$; affine planes of order $q$, which are Steiner systems $S(2,q,q^2)$; \emph{Steiner triple systems}, denoted $\STS(n)$, which are Steiner systems $S(2,3,n)$; and {\em Steiner quadruple systems}, denoted $\mathrm{SQS}(n)$, which are Steiner systems $S(3,4,n)$.  We will be interested in Steiner systems $S(k-1,k,n)$.

It is straightforward to show that for a Steiner triple system to exist, we must have $n\equiv 1,3 \pmod 6$; we call $n$ the {\em order} of the Steiner triple system.  The number of blocks in an $\STS(n)$ is $n(n-1)/6$.  In 1847, Kirkman~\cite{Kirkman} showed that Steiner triple systems exist for all admissible values of $n\geq 7$.  The unique $\STS(7)$ is the Fano plane.  Also, it is known that Steiner quadruple systems exist if and only if $n\equiv 2,4 \pmod 6$ (see~\cite[Theorem II.5.24]{handbook}).  Unfortunately, no existence result is known for Steiner systems in general; a table of parameters of known Steiner systems can be found in~\cite[Table II.5.17]{handbook}.  Very few Steiner systems are known for $k\geq 5$.

Given a Steiner system $S(k-1,k,n)$ and an $(k-1)$-subset of $[n]$, its \emph{completion} is defined to be the unique block in the system that contains that subset.  For example, in a Steiner triple system $\STS(n)$, one can complete any pair of elements to a unique block.



The main result of this subsection is as follows.
\begin{thm}
\label{thm:Steiner systems}
Suppose there exists a Steiner system $S(k-1,k,n)$, where $n\geq 4k-2$.  Then its blocks form a resolving set for $K(n,k)$, and consequently
\[ \beta(K(n,k))\leq \frac{1}{k} \binom{n}{k-1}. \]
\end{thm}

\begin{proof}
Let $\mathcal{S}$ be a Steiner system $S(k-1,k,n)$.  Suppose $U$ and $W$ are two distinct vertices of $K(n,k)$, and let $a\in U\setminus W$.
We can assume that $U$, $W$ are not blocks of~$\mathcal{S}$.  Now, one can choose a set of $k-3$ points $Y=\{y_1,\ldots, y_{k-3}\}$ disjoint from $U\cup W$, and then a further $k+1$ points $x_1,\ldots ,x_{k+1}\in [n]\setminus (U\cup W \cup Y)$.  For each $x_i$, form the completions of the $(k-1)$-subsets 
$\{a\}\cup \{y_1,\ldots, y_{k-3}\} \cup \{x_i\}$: these are blocks of $\mathcal{S}$ formed by including an additional point $b_i$.  Since $\mathcal{S}$ is a Steiner system, it follows that each of the elements $b_1,\ldots,b_{k+1}$ are distinct (otherwise, $\{a\}\cup \{y_1,\ldots, y_{k-3}\} \cup \{b_i\}$ would be a subset of more than one block).  By the pigeonhole principle, at least one of these elements $b_j$ is not in $W$.  Consequently, the block $X= \{a\}\cup \{y_1,\ldots, y_{k-3}\} \cup \{x_j\} \cup \{b_j\}$ is disjoint from $W$ but not $U$, and so $d(X,U)=1$ while $d(X,W)\neq 1$.  Hence $X$ resolves $U$ and $W$.

The bound follows from evaluating Equation~\ref{blocksnumber} in the case where $\lambda=1$ and $t=k-1$.
\end{proof}

In particular, in the special case of Steiner triple systems, our result has the following form.

\begin{cor}
\label{sts}
Let $n$ be an integer such that $n\equiv 1,3 \pmod 6$ and $n\geq 13$, and let $\mathcal{S}$ be a Steiner triple system of order $n$. Then the blocks of $\mathcal{S}$ form a resolving set for $K(n,3)$.  Consequently, $\beta(K(n,3))\leq n(n-1)/6$.
\end{cor}
%
We remark that this result does not include the Steiner triple systems of orders~7 and 9.  However, the unique $\STS(7)$ is the Fano plane, which we know from the previous subsection to be a resolving set for $K(7,3)$.  Also, the unique $\STS(9)$ is an affine plane, which we know from Corollary~\ref{affine} to be a resolving set for $K(9,3)$.


\subsection{Toroidal grids}
\label{sec:toroidal grids}

In this subsection, we obtain resolving sets for Kneser graphs $K(n,4), K(n,5), K(n,6)$, provided $n$ is sufficiently large, by using a construction in toroidal grids.  Although this construction does not apply directly to $K(n,k)$ where $k\geq 7$, we suspect that similar ideas could be developed to cover other values of the parameter $k$.

A \emph{toroidal grid} is the graph $H=C_a\square C_b$ with $n=ab$ vertices obtained as the cartesian product of two cycles, $C_a$ and $C_b$, with $a$ and $b$ vertices respectively.
A \emph{straight path} in $H$ is a set of vertices such that all of them share the first coordinate, or all of them share the second coordinate. If $x$ is a vertex of $H$ and $k\geq 1$, there are four straight paths with $k$ vertices having $x$ as an end-point: we will denote these as \emph{$(x,k)$-paths}. Fixing a cyclic ordering of the vertices of the cycles $C_a$ and $C_b$, we can say that an $(x,k)$-path in $H$ goes \emph{right} if the first coordinates of its vertices, beginning on vertex $x$, increase (thus the second coordinates are equal, by definition of straight path). Analogously, the path goes \emph{up} if the second coordinates, beginning on $x$, increase.  In a similar manner, we can describe $(x,k)$-paths going \emph{down} or \emph{left}.  Note that the total number of straight paths in $C_a \square C_b$ is $2ab$.

\begin{thm}
\label{4}
Let $H$ be the toroidal grid $C_a \square C_b$ with $a,b\geq 10$. If $n=ab$, the set of all straight paths in $H$ with 4 vertices is a resolving set for $K(n,4)$.  Therefore, for such values of $n$, we have $\beta(K(n,4))\leq 2n$.
\end{thm}

\begin{proof}
We identify the set $[n]$ with the $ab$ vertices of $H$, so a vertex of $K(n,4)$ is simply a 4-subset of $V(H)$.  Consider two such subsets $U,W$ with $U\neq W$.  We will show that there exists a straight path in $H$ that meets just one of them.

Since $U\neq W$, there exists $x$ such that $x \in U\setminus W$.  Denote by $p_1,p_2,p_3,p_4$ the $(x,4)$-paths which go right, up, left and down respectively.  If there exists $p_i$ such that $p_i\cap W=\emptyset$, we are done, so suppose that $p_i \cap W\neq \emptyset$ for $i=1,2,3,4$.  Note that $p_i\cap p_j=x\ (i\neq j)$, so it is clear that $p_i \cap W=\{y_i\}$ with $y_i\neq y_j\ (i\neq j)$ and therefore that the set $W$ must be $W=\{y_1,y_2,y_3,y_4\}$ (see Figure~\ref{caso1:subfig1}).

\begin{figure}[!ht]
\centering
\subfigure[Vertices of $W$ (in black) and the $(x,4)$-paths.]{
   \includegraphics[width=0.25\textwidth] {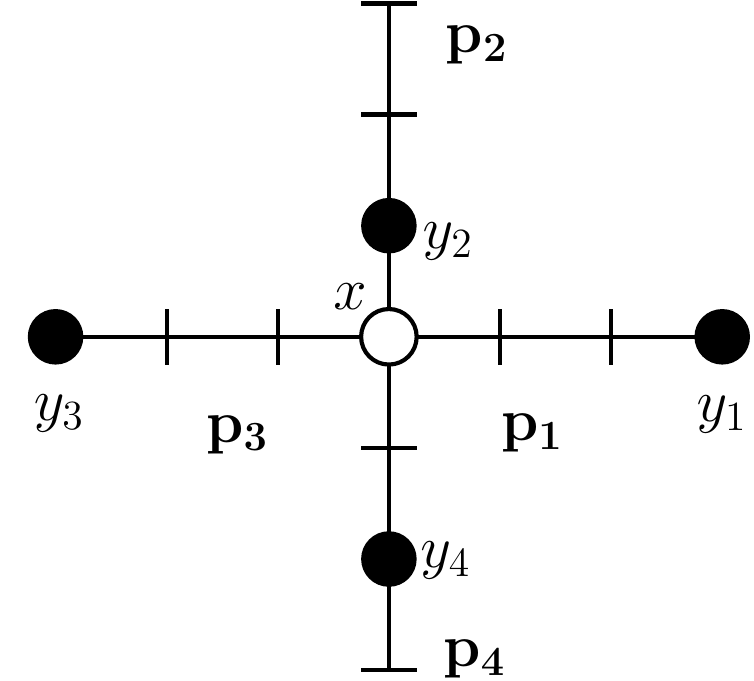}
   \label{caso1:subfig1}}\hspace{1cm}
   \subfigure[ The white vertices are in $U$ and black vertices are in $W$. The $y_2$-path meets $W$ but does not meet $U$.]{
   \includegraphics[width=0.37\textwidth] {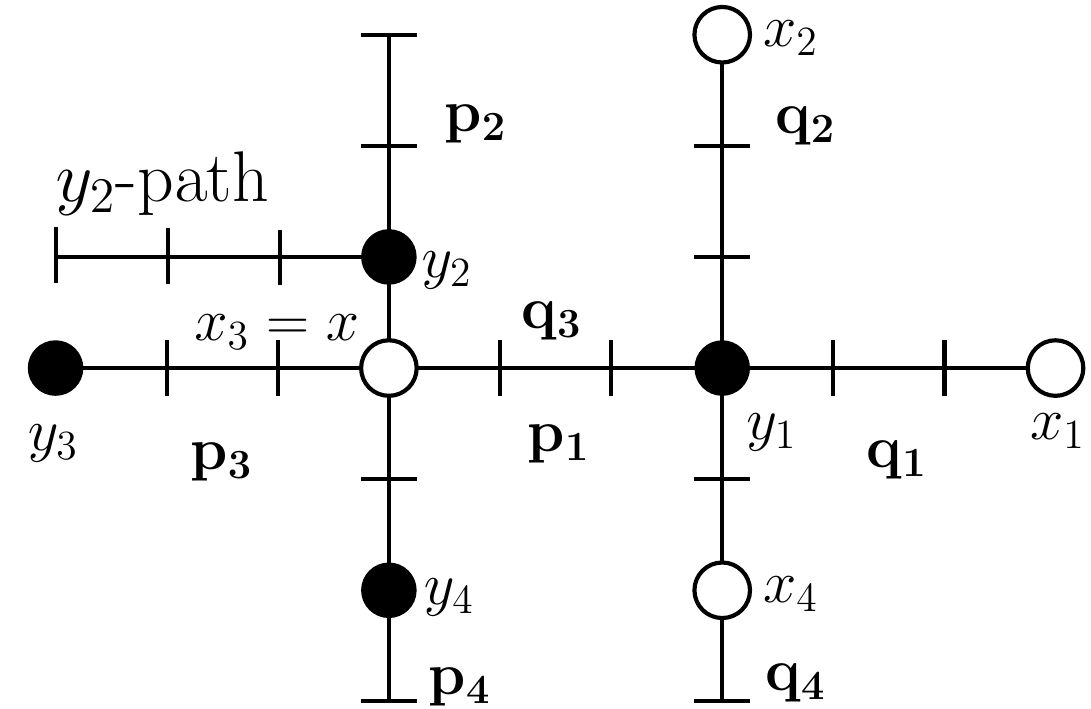}
   \label{caso1:subfig2}}
\label{caso1}
\caption{Construction for $K(n,4)$.}
\end{figure}

Assume, without lost of generality, that $y _1\notin U$.  Denote the $(y_1,4)$-paths as $q_1,q_2,q_3,q_4$, going right, up, left and down respectively. Note that $x$ belongs to $q_3$. If there exists $q_i$ such that $q_i\cap U=\emptyset$, we are done, so suppose that $q_i\cap U\neq \emptyset$ for $i=1,2,3,4$. Again, $q_i \cap U=\{x_i\}$ with $x_i\neq x_j\ (i\neq j)$ (note that $a,b\geq 10$), and $U=\{x_1,x_2,x=x_3,x_4\}$ (see Figure~\ref{caso1:subfig2}).

Now we observe that $y_2\notin U$ and the $(y_2,4)$-path going left meets $W$ but it does not meet $U$ (see Figure~\ref{caso1:subfig2}).  Hence this path has the desired property.
\end{proof}

In a similar way, given any pair of distinct 5-subsets (or of 6-subsets) of $V(H)$, we can find a straight path with 5 vertices (or with 6 vertices, respectively) that meets just one of them, provided that the toroidal grid $H$ is large enough.  So we obtain the following results about the metric dimension of $K(n,5)$ and $K(n,6)$.

\begin{thm}
\label{5}
Let $H$ be the toroidal grid $C_a \square C_b$ with $a,b\geq 13$.  If $n=ab$, the set of all straight paths in $H$ with 5 vertices is a resolving set for $K(n,5)$.  Therefore, for such values of $n$, we have $\beta(K(n,5))\leq 2n$.
\end{thm}

\begin{thm}
\label{6}
Let $H$ be the toroidal grid $C_a \square C_b$ with $a,b\geq 16$.  If $n=ab$, the set of all straight paths in $H$ with 6 vertices is a resolving set for $K(n,6)$.  Therefore, for such values of $n$, we have $\beta(K(n,6))\leq 2n$.
\end{thm}

Unfortunately, this construction using straight paths in a toroidal grid does not work for $K(n,k)$ with $k\geq 7$. In these cases, there are vertices in the Kneser graph that cannot be resolved using such paths (see Figure~\ref{caso4}).

\begin{figure}[!hp]
\centering
\includegraphics[width=0.2\textwidth]{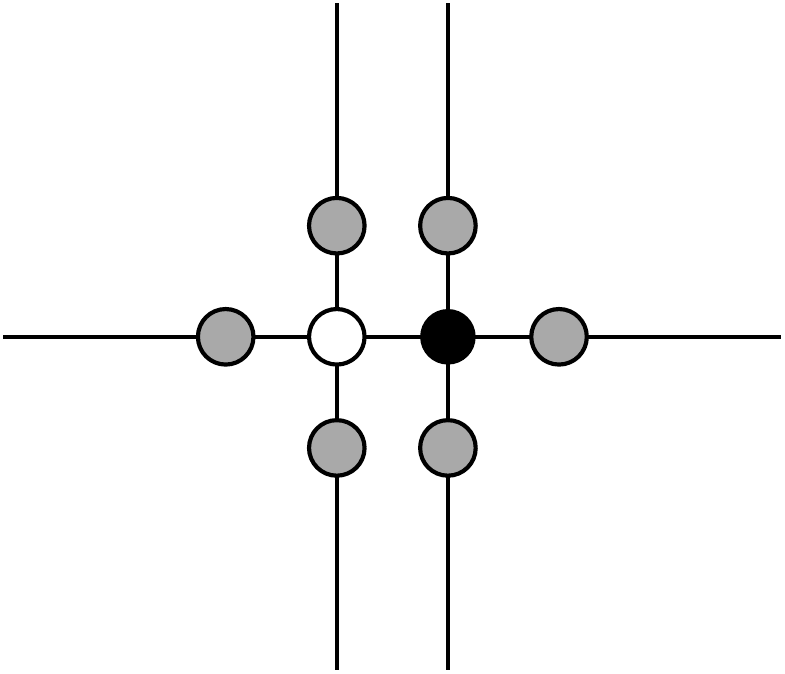}
\caption{The white vertex is in $U$, black vertices are in $W$ and grey vertices are in both of them. Vertices $U,W\in V(K(n,7))$ cannot be resolved using straight 7-paths.}
\label{caso4}
\end{figure}

\section{Final remarks}
\label{sec:final}

In this paper, our emphasis has been on finding constructions of resolving sets for Johnson and Kneser graphs, using various algebraic, combinatorial and geometric approaches.  Nevertheless, these constructions provide bounds on the the metric dimension of $J(n,k)$ and $K(n,k)$.  We summarize these bounds in Table~\ref{summarytable}.  In the table, $q$ denotes a prime power.
\renewcommand{\arraystretch}{1.4}
\begin{table}[p]
\centering
{\small
\begin{tabular}{|c|c|c|}
\hline
{\bf The metric dimension of ...} & {\bf using ...} & {\bf is bounded by ...} \\
\hline
$J(n,2)$, $K(n,2)$ & \multirow{2}*{\cite[Corollary 3.33]{bsmd}} & \multirow{2}*{${\frac{2}{3}}(n-i)+i$} \\
where $n\equiv i \pmod 3$ & & \\
\hline
$J(n,k)$ & \multirow{4}*{partitioning $[n]$} & $\lfloor k(n+1)/(k+1) \rfloor$ \\
\cline{1-1}\cline{3-3}
$K(2k+1,k)=O_{k+1}$  &                       & $2k$ \\
\cline{1-1}\cline{3-3}
$K(n,k)$ &                                   & $\lceil \frac{n}{2k-1} \rceil (\binom{2k-1}{k}-1)$ \\
\cline{1-1}\cline{3-3}
$K(n,k)$, diameter~3 &                       & $2\binom{n-k}{k}$ \\
\hline
\multirow{3}*{$J(n,k)$} & $k$-set system whose & \multirow{3}*{$n$} \\
                        & incidence matrix has rank $n$  & \\
\cline{2-2}
                        & $(n,k,\lambda)$ symmetric design & \\                        
\hline
$J(q^2+q+1,q+1)$ & projective plane of order $q$ & $q^2+q+1$ \\
\hline 
$J(4m-1,2m-1)$,       & \multirow{2}*{Hadamard design} & \multirow{2}*{$4m-1$} \\
$K(4m-1,2m-1)=O_{2m}$ & & \\
\hline
$K(n,3)$ & Steiner triple system $\STS(n)$ & $n(n-1)/6$ \\
\hline
$K(n,k)$ & Steiner system $S(k-1,k,n)$ & $\binom{n}{k-1}/k$ \\
\hline
$K(v,s+1)$,                 & \multirow{2}*{partial geometry $pg(s,t,\alpha)$} & \multirow{2}*{$(t+1)(st+\alpha)/\alpha$} \\
$v=(s+1)(st+\alpha)/\alpha$ & & \\
\hline
$K(q^2,q)$ & affine plane of order $q$ & $q(q+1)$ \\
\hline
$K(q^3,q)$ & \multirow{3}*{generalized quadrangle} & $q^2(q+2)$\\
\cline{1-1}\cline{3-3}
$K((q+1)(q^3+1),q+1)$  & & $(q^2+1)(q^3+1)$ \\
\cline{1-1}\cline{3-3}
$K((q^2+1)(q^5+1),q^2+1)$ & & $(q^3+1)(q^5+1)$\\
\hline
$K(n,4)$ & \multirow{3}*{toroidal grid $C_a\Box C_b$} & \multirow{3}*{$2ab=2n$}\\
\cline{1-1}
$K(n,5)$ &     &  \\
\cline{1-1}
$K(n,6)$ &   &  \\
\hline
\end{tabular}}
\caption{A summary of bounds on the metric dimension of $J(n,k)$ and $K(n,k)$} \label{summarytable}
\end{table}

Note that many of the bounds in Table~\ref{summarytable} are conditioned on the existence of certain objects, or require parameters to be sufficiently large.  However, we expect that these bounds hold more widely.  In particular, we conjecture that for the Kneser graph $K(n,k)$ there is a bound of $\beta(K(n,k))=\mathrm{O}(n)$ independent of $k$.

Also, the question of determining the exact values of $\beta(J(n,k))$ and $\beta(K(n,k))$ remains open for $k>2$.  This is likely to be quite challenging in general.  As part of our investigations, we conducted some computer searches using the \textsf{GAP} system~\cite{GAP}.  One pattern that emerged was that, for the Johnson graph $J(2k,k)$, the metric dimension appeared to equal $k+1$: we also conjecture that this is the exact value.

\end{document}